\documentclass[10pt]{amsart}

\usepackage[leqno]{amsmath}
\usepackage{amsthm}
\usepackage{amsfonts}
\usepackage{amssymb}
\usepackage{eucal}
\usepackage{graphicx}
\DeclareGraphicsExtensions{.pdf,.png,.jpg}
\usepackage{hyperref}
\usepackage[all]{xy}
\CompileMatrices

\newcommand{\tuple}[1]{\ensuremath{\left \langle #1 \right \rangle }}

\DeclareFontFamily{OT1}{rsfs}{}
\DeclareFontShape{OT1}{rsfs}{n}{it}{<-> rsfs10}{}
\DeclareMathAlphabet{\mathscr}{OT1}{rsfs}{n}{it}

\theoremstyle{plain}
  \newtheorem{theorem}{Theorem}

\theoremstyle{definition}

\theoremstyle{remark}

\numberwithin{equation}{section}

   \DeclareMathOperator{\Elem}{\operatorname{Elem}}
   \DeclareMathOperator{\HP}{\operatorname{HP}}
   \DeclareMathOperator{\RB}{\operatorname{ReplaceBase}}
   \DeclareMathOperator{\rem}{\operatorname{rem}}

\title{Exponential prefixed polynomial equations}
\author{Aran Nayebi}
\date{September 12, 2012}
\email{anayebi@stanford.edu}
\urladdr{http://www.stanford.edu/~anayebi}
\begin{document}
\begin{abstract}
A prefixed polynomial equation is an equation of the form $P(t_1,\ldots,t_n) = 0$, where $P$ is a polynomial whose variables $t_1,\ldots,t_n$ range over the natural numbers, preceded by quantifiers over some, or all, of its variables. Here, we consider exponential prefixed polynomial equations (EPPEs), where variables can also occur as exponents.  We obtain a relatively concise EPPE equivalent to the combinatorial principle of the Paris-Harrington theorem for pairs (which is independent of primitive recursive arithmetic), as well as an EPPE equivalent to Goodstein's theorem (which is independent of Peano arithmetic). Some new devices are used in addition to known methods for the elimination of bounded universal quantifiers for Diophantine predicates.
\end{abstract}
\maketitle
\tableofcontents

\section{Introduction}
A prefixed polynomial equation is an equation of the form $P(t_1,\ldots,t_n) = 0$, where $P$ is a polynomial with variables $t_1,\ldots,t_n$ that range over the natural numbers, preceded by quantifiers over some or all of its variables. Bovykin and De Smet \cite{bovykin} study the collection of all such possible expressions (terming this ``the Atlas''), with the equivalence of relation of being ``EFA-provably equivalent'' on its members. Thus, members of the same class are prefixed polynomial expressions that are provably equivalent to one another. It is not difficult to obtain a prefixed polynomial representation, but the value of obtaining polynomial expressions is that they provide concrete examples of unprovable statements and explicit illustrations of deep logical phenomena. For example, the prefixed polynomial expression that Bovykin and De Smet obtain for 1-Con(ZFC+Mahlo) implies, over $I\Sigma_1$, all two quantifier arithmetical theorems that can be proved in ZFC + Mahlo cardinals. To avoid too much repetition, we refer the reader to the discussion in \cite{bovykin} for a detailed exposition as to the importance of such an Atlas. \newline
\indent One example of prefixed polynomial expressions of well-known logical phenomena that Bovykin and De Smet obtain are representations (involving alternations of universal and existential quantifiers) of the Paris Harrington theorem, and the special cases of the Paris-Harrington theorem for pairs and triples. The Paris-Harrington theorem \cite{paris-harrington}, which states that a simple extension to the finite Ramsey theorem is not provable in first-order Peano arithmetic, is a seemingly natural mathematical example of incompleteness, namely:
\begin{theorem}[PH]
For all numbers $e, r,$ and $k$, there exists a number $M$, such that for every coloring $f$ of $e$-subsets of $[M+1] = \{0, 1, \ldots, M\}$ into $r$ colors, there is an $f$-homogeneous $Y \subseteq [M+1]$ of size at least $\operatorname{min}(Y) + k - 1$.
\end{theorem}
Our focus will be on the Paris-Harrington theorems for pairs, whose combinatorial principle will be referred to here on out as PH$^2$:
\begin{theorem}[PH$^2$] \label{ph2}
For every number $k$, there exists a number $M$ such that for every coloring $f$ of 2-subsets of $[M+1]$ into $r$ colors, there is a $f$-homogeneous $Y \subseteq [M+1]$ of size at least $\operatorname{min}(Y) + k - 1$.
\end{theorem}
For PH$^2$ and $r > 2$, (where $r$ is the only free variable which represents the colors), the Bovykin-De Smet prefixed polynomial representation \cite[Theorem 2]{bovykin} is:
\begin{equation} \label{ppr1}
\begin{split}
& \forall k \mbox{            } \exists M \mbox{            } \forall ab \mbox{            } \exists cdAX \mbox{            } \forall xy \mbox{            } \exists BCF \mbox{            } \forall fg \mbox{            }\exists ehilnpq \\
& [x \cdot (y + B - x) \cdot (A+k+B-y) \cdot ((((f-A)^2 (g-1)^2)\cdot((f-B)^2 + (g-x)^2)\\
&\cdot((f-C)^2 + (g-y)^2)-h-1)\cdot((dgi + i -c+f)^2 + (f+h-dg)^2)+(B+l+1-C)^2 \\
&+ (C+n-M)^2 + (F+e-b(B+C^2))^2+(bp(B+C^2)+p-a+F)^2+((F-X)^2-qr)^2)].
\end{split}
\end{equation}
Although in the case of PH$^2$, the prefixed polynomial expression covers only a few lines, the challenge comes when transforming this polynomial from its $\Pi^0_6$ form to its EFA-provably equivalent $\Pi^0_2$ form. Here, bounding the universal quantifiers and then eliminating them introduces a drastic increase in the number of variables of the original prefixed polynomial representation, to the point that the resulting $\Pi^0_2$ form is too long to be practical to write. The transformation of formulas containing bounded universal quantifiers into equivalent formulas containing only existential quantifiers in the theory of Diophantine equations is a powerful technique which has many applications, such as showing in a straightforward manner that the set of primes is Diophantine, constructing a universal Diophantine equation, or demonstrating that many famous problems can be reformulated in terms of the unsolvability of a particular Diophantine equation (since many of these problems such as Goldbach's conjecture, the Riemann hypothesis, and the four color theorem can be formulated in the form $\forall n \mbox{            } P(n)$, where $P$ is a decidable property over natural numbers). However, naive attempts to obtain a Diophantine representation (namely, a direct application of the results of Davis, Putnam, and Robinson \cite{dpr} and Matiyasevich \cite{mat1970}, and possibly with some slight modifications but with no drastic tricks) for PH$^2$ yields unwriteable representations. \newline
\indent We discuss the methods that are used to eliminate the bounded universal quantifier and present several ways of conserving the large number of variables typically introduced by this process in order to obtain the following result:
\begin{theorem}[Unprovability by primitive recursive means] \label{result1}
There exists an exponential Diophantine equation $E_1(k, M, a, b, r, t_1,\ldots,t_{138})$ and a Diophantine equation $D_1(k, M, a, b, r, t_1, \ldots, t_{347})$ (both with $k, M, a, b,$ and $r$ as parameters) such that for every $r > 2$
\begin{equation*}
E_1(k, M, a, b, r, t_1,\ldots,t_{138}) = 0
\end{equation*}
has a solution in natural numbers $t_1,\ldots,t_{138}$ and
\begin{equation*}
D_1(k, M, a, b, r, t_1, \ldots, t_{347}) = 0
\end{equation*}
has a solution in natural numbers $t_1,\ldots,t_{347}$ is equivalent to the combinatorial principle of the Paris-Harrington theorem for pairs, equivalent to the 1-consistency of $I\Sigma_1$ and thus not provable in $I\Sigma_1$ (but provable in $I\Sigma_2$).
\end{theorem}
We will prove Theorem ~\ref{result1} in \S 3.1 and provide an explicit representation of $E_1(k, M, a, b, r, t_1, ..., t_{138})$. In \S 3.2, we consider unprovability in Peano arithmetic and obtain an explicit EPPE equivalent to Goodstein's theorem, obtaining the following result:
\begin{theorem}[Unprovability in Peano arithmetic] \label{result2}
There is a 181 variable exponential prefixed polynomial equation equivalent to Goodstein's theorem, unprovable in Peano arithmetic.
\end{theorem}
 Our results offer insight in preserving a writeable representation when we reduce the quantifier complexity of the original prefixed polynomial representations of Bovykin and De Smet \cite{bovykin}, therefore not restricting us to using alternations of universal and existential quantifiers in order to explicitly illustrate deep logical phenomena. Not only that, but our consideration of exponential prefixed polynomial equations allows one to obtain short representations of ``natural'' independent statements for which exponentiation is inherent in the formulation, such as Goodstein's theorem.
\section{Elimination of the bounded universal quantifier}
There are several methods of transforming formulas with bounded universal quantifiers to those having only existentially bound variables in the language of Diophantine predicates. The most well-known is the Bounded Quantifier Theorem of Davis, Putnam, and Robinson \cite{dpr}, which uses the Chinese remainder theorem to establish this equivalence. Bounded universal quantifiers can also be eliminated by way of Turing machines (presented in detail in Chapter 6.1 of Matiyasevich \cite{mat-book}), a rather immediate consequence of Matiyasevich's direct method in \cite{mat76} of simulating Turing machines by Diophantine equations. Finally, a third method of going about this elimination involves summations of generalized geometric progressions, based on a technique first proposed by Matiyasevich in \cite{mat79} and presented for the first time in Chapter 6.3 of Matiyasevich's book \cite{mat-book}. There are advantages and disadvantages to each method. The method via Turing machines, though constructive, is rather roundabout. With the method involving summations, although straightforward, it becomes impractical to extract the resultant Diophantine equation from the heavy use of generalized geometric progressions after the elimination of the bounded universal quantifier when the expression obtained \emph{prior} to this elimination is of even moderate size. Thus, this is our motivation for preferring the method of Davis, Putnam, and Robinson via the Chinese remainder theorem because it is a straightforward number-theoretic method which produces, in most cases, a visualizable Diophantine equation. However, as we will explicitly demonstrate below, this does not mean that the representation can be, practically speaking, explicitly written down since the downside is the drastic number of variables introduced. \newline
\indent For example, for the following:
\begin{equation} \label{eq1}
\forall y < b \mbox{          } \exists x_1,\ldots,x_m [G(\underline{a},y,x_1,\ldots,x_m) = 0],
\end{equation}
where $\underline{a}$ represents the parameter(s) of the polynomial $G$, the Chinese remainder theorem method results in the following system of Diophantine conditions solvable in the unknowns $q, w, z_0,\ldots,z_m$ provided that \eqref{eq1} holds\footnote{This is based on the original construction of the Bounded Quantifier Theorem of Davis, Putnam, and Robinson \cite{dpr}, with some minor modifications by Matiyasevich \cite[Ch. 6.2]{mat-book}. There have been more drastic modifications to the Bounded Quantifier Theorem, namely by Matiyasevich \cite{mat72}\cite{mat73} and by Hirose and Iida \cite{hi73}, and we incorporate some of these results in our presentation.}:
\begin{equation} \label{eq2}
\begin{split}
&G(\underline{a},z_0,z_1,\ldots,z_m) \equiv 0 \mod \binom{q}{b}, \\
& z_0 = q, \\
& b!(b+w+B(\underline{a},b,w))! \mid q+1,
\end{split}
\end{equation}
\begin{equation} \label{eq3}
\begin{split}
& \binom{q}{b} \mid \binom{z_1}{w}, \\
& \hspace{0.31in} \vdots \\
& \binom{q}{b} \mid \binom{z_m}{w},
\end{split}
\end{equation}
where the polynomial $B(\underline{a},b,w)$ is obtained from $G(\underline{a},y,x_1,\ldots,x_m)$ by changing the signs of all its negative coefficients and systematically replacing $y$ by $b$ and $x_1,\ldots,x_m$ by $w$. \newline
\indent Now, the major contributing factors to the increase in the number of variables are the representations of the factorial and binomial coefficient. Let $\Gamma$ denote the number of variables in the Diophantine representation of the exponential function, then the Diophantine representations of the factorial and binomial coefficient as presented by Davis, Matiyasevich, and Robinson \cite[\S 1]{dmr1976} involve $10+5(\Gamma+1)$ variables and $6+3(\Gamma+1)$ variables, respectively. And the Diophantine representations of the factorial and binomial coefficient as later presented by Matiyasevich \cite[Chapter 3.4]{mat-book} involve $10+6(\Gamma+1)$ variables and $5 + 4(\Gamma+1)$ variables, respectively. If we used an economical (with respect to the number of variables) Diophantine representation of the exponential function, for instance, the result obtained by Matiyasevich and Robinson \cite{mat-rob1975} with only five variables, then the total number of variables for the factorial and binomial coefficient presented in \cite{dmr1976} would be 40 variables and 24 variables, respectively; and the results obtained in \cite{mat-book} would be 46 variables and 29 variables, respectively. Now, note that the representation of the exponential function in five variables obtained in \cite{mat-rob1975} results in a polynomial of high degree and is a somewhat unruly expression (though obviously writeable), for both the binomial coefficient and the factorial. Furthermore, it is important to note that \eqref{eq3} involves $m$ binomial coefficients (I am excluding $\binom{q}{b}$). Thus, even if we used the least number of variables in representing the binomial coefficient, 25 variables (since for each binomial coefficient, we have to introduce a new variable $y_l$, such that $y_l = \binom{z_l}{w}$, for each $l = 1,\ldots,m$), then \eqref{eq3} would require $25m$ variables to represent it. For $m = 3$, a possibly small Diophantine equation with a single universally bound variable $y$, this would mean the introduction of 75 new variables just to represent the system of conditions in \eqref{eq3}!\newline
\indent Thus, the first step toward mitigating the increase of variables introduced by \eqref{eq2} and \eqref{eq3} the number of variables involved in the Diophantine representations of the factorial and the binomial coefficient must be drastically reduced. Fortunately, it turns out that one can eliminate the need to show that factorial is Diophantine. By a result of Matiyasevich \cite{mat72}\cite{mat73}, one can use the multiplicative version of Dirichlet's box principle to replace the condition in \eqref{eq2}:
\begin{equation} \label{save1}
b!(b+w+B(\underline{a},b,w))! \mid q+1
\end{equation}
with the sufficiently strong inequality,
\begin{equation} \label{save2}
q > b + (b+1)^{b+1}((b+1)^{b+1}B(\underline{a},b,w))^{w^m}.
\end{equation}
As can easily be seen, \eqref{save1} would require the introduction of $10(\Gamma+1) + 22$ variables (using the representation of the factorial provided by Davis, Matiyasevich, and Robinson \cite[\S 1]{dmr1976}) whereas \eqref{save2} would only require the introduction of $2(\Gamma+1)+2$ variables. Hence, even if one is content with just an exponential Diophantine representation, then 20 variables are already conserved, and if one would prefer a Diophantine representation even with an economical representation of the exponential function in only five variables, then 68 variables are conserved! \newline
\indent However, what has not really been proposed so far is a reduction in the number of variables introduced by \eqref{eq3}, the system of binomial coefficients, since that is the \emph{primary} reason why so many variables are introduced in the elimination of the bounded universal quantifier (as demonstrated above in the case for $m = 3$). We will prove that one can represent the binomial coefficient in only 10 variables and obtain an explicit representation that covers less than a page at 18 variables. \newline
\indent In our construction of a representation of the binomial coefficient in only 10 variables, we will rely on the relation-combining theorem of Matiyasevich and Robinson \cite{mat-rob1975}. The Matiyasevich-Robinson relation-combining theorem allows one to cheaply define certain combinations of relations than by defining each separately by an equation and then combining the equations. While it is economical with respect to the number of variables of the resultant equation, it should be noted that the relation-combining theorem is rather uneconomical with respect to the degree\footnote{A version of the relation-combining theorem that is more efficient with respect to the degree was later worked out by Matiyasevich, presented as Theorem 5.1 of Jones \cite{jones-mersenne}. However, this does not change the fact that the resultant polynomial can cover several pages, which is why we rely on more elementary techniques to have our representation of the binomial coefficient cover half a page, with an introduction of only eight more variables.}. The theorem is stated as follows:
\begin{theorem}[Relation-combining theorem] \label{l1}
Let $\Box$ denote a perfect square. For all integers $A_1,\ldots,A_q, B, C, D$ with $B \ne 0$, the conditions $A_i = \Box \mbox{               } (i = 1,\ldots,q)$, $B \mid C$, and $D > 0$ all hold if and only if $M_q(A_1,\ldots,A_q,B,C,D,n) = 0$ for some $n$, where $M_q$ is the following $2^q$-fold product over all combinations of signs
\begin{equation*}
M_q = \prod (B^2n + C^2 - B^2(2D-1)\cdot(C^2+W_q\pm\sqrt{A_1}\pm\sqrt{A_2}W\pm\ldots\pm\sqrt{A_q}W^{q-1})),
\end{equation*}
where 
\begin{equation*}
W = 1 + \sum_{i = 1}^{q}A_i^2.
\end{equation*}
\end{theorem}
We also need a result on the solutions to the Pell equation
\begin{equation} \label{pell}
x^2 - (a^2-1)y^2 = 1.
\end{equation}
For $a > 0$, we define the pair $<\chi_{a}(n),\psi_{a}(n)>$ as the $n$-th nonnegative solution of \eqref{pell}. Theorem 4 of Matiyasevich and Robinson \cite{mat-rob1975} proves the following system of Diophantine conditions:
\begin{theorem} \label{l2}
For $A > 1$, $B > 0$, and $C > 0$, $C = \psi_{A}(B)$ if and only if the following system of conditions is satisfied:
\begin{equation*}
\begin{split}
& DFI = \Box, \mbox{            } F \mid H-C, \mbox{            } B \le C, \\
& D = (A^2 - 1)C^2 + 1, \\
& E = 2(i + 1)DC^2, \\
& F = (A^2 - 1)E^2 + 1, \\
& G = A + F(F-A), \\
& H = B + 2jC, \\
& I = (G^2 - 1)H^2 + 1.
\end{split}
\end{equation*}
\end{theorem}
From Theorems ~\ref{l1} and ~\ref{l2}, we are led to the following result:
\begin{theorem} \label{t2}
The relation $y = \binom{n}{s}$, where $n \ge s > 0$, holds if and only if
\begin{equation*}
F(y, n, s, x, w, k, l, m, i, j, v_1, v_2, v_3) = 0
\end{equation*}
has a solution in the parameters $y, n$, and $s$ and the variables $x, w, k, l, m, i, j, v_1, v_2$, and $v_3$.
\end{theorem}
\begin{proof}
It is easy to see that
\begin{equation} \label{8}
\begin{split}
& y = \binom{n}{s} \\
& \Leftrightarrow \exists x\Big[y = \operatorname{rem}\Big(\Big[\frac{(x+1)^n}{x^s}\Bigr],x\Bigr) \land x > 4n^s\Bigr]\\
& \Leftrightarrow \exists xw\Big[y = \operatorname{rem}(w+1,x) \land w+1 = \Big[\frac{(x+1)^n}{x^s}\Bigr]  \land x - 4n^s > 0\Bigr],
\end{split}
\end{equation}
where the partial binomial expansion $\Big[\frac{(x+1)^n}{x^s}\Bigr]$ is defined as
\begin{equation*}
\Big[\frac{(x+1)^n}{x^s}\Bigr] = \sum_{i = 0}^{n-s}\binom{n}{s+i}x^i,
\end{equation*}
for $n > 0$, $s > 0$, and $x > n^s$. Note that the reason we take $w + 1$ as opposed to just $w$ in \eqref{8} is to ensure that $w + 1 > 0$. By Theorem 8 of \cite{mat-rob1975}, $w+1 = \Big[\frac{(x+1)^n}{x^s}\Bigr]$ can be expressed as a system of Diophantine conditions in three variables, so \eqref{8} becomes the following system of equations
\begin{equation*}
\begin{split}
& y = \operatorname{rem}(w+1,x), \\
& C = \psi_A(B), \\
& (M^2 - 1)K^2 + 1 = \Box, \\
& (M^2x^2 - 1)L^2 + 1 = \Box, \\
& (x - 4n^s)(K^2L^2 - 4(C-KL(w+1))^2) > 0, \\
& M = 8n(x+w+1)+2, \\
& K = n - s + 1 + k(M-1), \\
& L = s + 1 + l(Mx-1), \\
& A = M(x+1), \\
& B = n + 1, \\
& C = m + B
\end{split}
\end{equation*}
Note that
\begin{equation*}
\begin{split}
& y = \operatorname{rem}(w+1,x) \\
& \Leftrightarrow (y+v_1+1-x)^2 + (xv_2-w-1+y)^2 = 0.
\end{split}
\end{equation*}
Moreover, since $C \ge B$, the remaining conditions can be combined via Theorem ~\ref{l1} as
\begin{equation*}
M_3(DFI, (M^2-1)K^2 + 1, (M^2x^2-1)L^2 + 1, F, H-C, (x - 4n^s)(K^2L^2 - 4(C-KL(w+1))^2), v_3) = 0.
\end{equation*}
Since we can eliminate $D, F, I, H, M, K, L, A, B,$ and $C$, we get that
\begin{equation*}
\begin{split}
& F(y, n, s, x, w, k, l, m, i, j, v_1, v_2, v_3) \\
& = ((y+v_1+1-x)^2 + (xv_2-w-1+y)^2)^2 \\
& \hspace{0.15in}+ M_3^2(DFI, (M^2-1)K^2 + 1, (M^2x^2-1)L^2 + 1, F, H-C, (x - 4n^s)(K^2L^2 - 4(C-KL(w+1))^2), v_3).
\end{split}
\end{equation*}
\end{proof}
The explicit representation of $F(y, n, s, x, w, k, l, m, i, j, v_1, v_2, v_3)$ is rather unruly and would cover several pages. Instead, we keep $D, F, I, K, L, M$ as variables and introduce two more, $W$ and $J$,  where $W = 1+(DFI)^2+((M^2 - 1)K^2 + 1)^2 + ( (M^2x^2 - 1)L^2 + 1)^2$ and $J = (x - 4n^s)(K^2L^2 - 4(C-KL(w+1))^2)$. Thus, we combine the additional equations via the summing of squares technique, and the following explicit representation of $F_1(y, n, s, x, w, k, l, m, i, j, v_1, v_2, v_3, D, F, I, J, K, L, M, W)$ is obtained, where $y, n$, and $s$ are parameters:
\begin{equation*}
\begin{split}
&((y+v_1+1-x)^2 + (xv_2-w-1+y)^2)^2 + ((D-((M(x+1))^2-1)(m+n+1)^2 - 1)^2 + (F-4((M(x+1))^2-1)\\
&(i+1)^2(((M(x+1))^2-1)(m+n+1)^2+1)^2(m+n+1)^4-1)^2 + (I-((M(x+1)+(4((M(x+1))^2-1)(i+1)^2\\
&(m+n+1)^4(((M(x+1))^2-1)(m+n+1)^2+1)^2)(4((M(x+1))^2-1)(i+1)^2(m+n+1)^4(((M(x+1))^2-1)\\
&(m+n+1)^2+1)^2-M(x+1)+1))^2-1)(n+1+2j(m+n+1))^2 - 1)^2 + (M-8n(x+w+1)-2)^2 + (K-n+s-1\\
&-k(M-1))^2 + (L-s-1-l(Mx-1))^2 + (W-1-(DFI)^2-((M^2-1)K^2+1)^2-((M^2x^2-1)L^2+1)^2)^2 \\
&+ (J - (x-4n^s)(K^2L^2-4(m+n+1-KL(w+1))^2))^2)^2 + ((F^2v_3 + (n+1+(m+n+1)(2j-1))^2 - F^2(2J-1)\\
&((n+1+(m+n+1)(2j-1))^2 + W^3 + (DFI)^{1/2} + ((M^2-1)K^2+1)^{1/2}W + ((M^2x^2-1)L^2+1)^{1/2}W^2))(F^2v_3 \\
&+ (n+1+(m+n+1)(2j-1))^2 - F^2(2J-1)((n+1+(m+n+1)(2j-1))^2 + W^3 + (DFI)^{1/2} - ((M^2-1)K^2+1)^{1/2}\\
& W + ((M^2x^2-1)L^2+1)^{1/2}W^2))(F^2v_3 + (n+1+(m+n+1)(2j-1))^2 - F^2(2J-1)((n+1+(m+n+1)(2j-1))^2 \\
& + W^3 + (DFI)^{1/2} - ((M^2-1)K^2+1)^{1/2}W - ((M^2x^2-1)L^2+1)^{1/2}W^2))(F^2v_3 + (n+1+(m+n+1)(2j-1))^2 \\
&- F^2(2J-1)((n+1+(m+n+1)(2j-1))^2 + W^3 - (DFI)^{1/2} - ((M^2-1)K^2+1)^{1/2}W - ((M^2x^2-1)L^2+1)^{1/2}W^2))\\
&(F^2v_3 + (n+1+(m+n+1)(2j-1))^2 - F^2(2J-1)((n+1+(m+n+1)(2j-1))^2 + W^3 - (DFI)^{1/2} - ((M^2-1)K^2\\
&+1)^{1/2}W + ((M^2x^2-1)L^2+1)^{1/2}W^2))(F^2v_3 + (n+1+(m+n+1)(2j-1))^2 - F^2(2J-1)((n+1+(m+n+1)\\
&(2j-1))^2 + W^3 - (DFI)^{1/2} + ((M^2-1)K^2+1)^{1/2}W + ((M^2x^2-1)L^2+1)^{1/2}W^2))(F^2v_3 + (n+1+(m+n+1)\\
&(2j-1))^2 - F^2(2J-1)((n+1+(m+n+1)(2j-1))^2 + W^3 - (DFI)^{1/2} + ((M^2-1)K^2+1)^{1/2}W \\
&- ((M^2x^2-1)L^2+1)^{1/2}W^2))(F^2v_3 + (n+1+(m+n+1)(2j-1))^2 - F^2(2J-1)((n+1+(m+n+1)(2j-1))^2 \\
&+ W^3 + (DFI)^{1/2} + ((M^2-1)K^2+1)^{1/2}W - ((M^2x^2-1)L^2+1)^{1/2}W^2)))^2.
\end{split}
\end{equation*}
\indent Now, if one would prefer to have a writeable exponential Diophantine equation from the transformation of a formula with bounded universal quantifier(s) (since for certain problems obtaining a writeable exponential Diophantine representation is more feasible than obtaining a writeable Diophantine representation, as will be the case in \S 3), then even the representation of the binomial coefficient in 18 variables is too large, since one can use the usual exponential Diophantine representations of these in 5 variables and 10 variables, respectively (if one uses the exponential Diophantine representations presented in \cite[Chapter 3.4]{mat-book}). However, the problem one still faces is the vast number of variables introduced by \eqref{eq3}. The relation of divisibility requires the introduction of one new variable, and each binomial coefficient (again, I am excluding $\binom{q}{b}$ from this) requires the introduction of six new variables. Thus, $7m$ new variables are needed in the exponential Diophantine representation of \eqref{eq3}. For example, in \S 3, we will be dealing with $m = 24$ and $m = 31$, which would mean for each case 168 and 217 variables are introduced, respectively. \newline
\indent Therefore, a method of reducing the number of variables involved in \eqref{eq3} would do us well. More specifically, the goal is to reduce the number of variables in
\begin{equation} \label{s1}
y_1 \mid \binom{z_l}{w},
\end{equation}
where $y_1 = \binom{q}{b}$ and $l = 1, \ldots, m$. \newline
\indent Inspired by some tricks used in the proof of the result by Matiyasevich \cite{mat79} that every Diophantine set has an exponential Diophantine representation with only three unknowns, the following result is obtained:
\begin{theorem} \label{t3}
\begin{equation} \label{s2}
\begin{split}
& y_1 \mid \binom{z_l}{w} \\
& \Leftrightarrow \exists pq \big[\left(({y_1}+1)2^{z_l} + 1\right)^{z_l} = p(({y_1}+1)2^{z_l})^{w} + q \land \left(({y_1}+1)2^{z_l}\right)^w > q \land y_1 \mid p\bigr].
\end{split}
\end{equation}
Hence, since the relations $\left(({y_1}+1)2^{z_l}\right)^w > q$ and $ y_1 \mid p$ introduce 2 more variables, we have a total of 4 new variables introduced by \eqref{s2}.
\end{theorem}
\begin{proof}
Observe that by the binomial theorem,
\begin{equation} \label{s3}
(u+1)^{z_l} = pu^w + q,
\end{equation}
where
\begin{equation*}
\begin{split}
& p = \sum_{i = w}^{z_l}\binom{z_l}{i}u^{i - w}, \\
& q = \sum_{i = 0}^{w-1}\binom{z_l}{i}u^{i}.
\end{split}
\end{equation*}
If $u$ is large enough, for instance, if
\begin{equation*}
u \ge 2^{z_l},
\end{equation*}
then
\begin{equation} \label{s4}
\begin{split}
u^w & \ge u^{w-1}2^{z_l} \\
& = \sum_{i = 0}^{z_l}\binom{z_l}{i}u^{w-1} \\
& > q.
\end{split}
\end{equation}
Note that $p$ and $q$ are uniquely determined by \eqref{s3} and \eqref{s4}. \newline
\indent It is easy to see that since
\begin{equation*}
p = \binom{z_l}{w} + \sum_{i = w+1}^{z_l}\binom{z_l}{i} u^{i-w},
\end{equation*}
then
\begin{equation*}
p \equiv \binom{z_l}{w} \mod{u}.
\end{equation*}
So if $y_1 \mid u$, then the condition in \eqref{s1} is equivalent to
\begin{equation*}
y_1 \mid p.
\end{equation*}
Thus,
\begin{equation*}
\begin{split}
& y_1 \mid \binom{z_l}{w} \\
& \Leftrightarrow \exists pqu [(u+1)^{z_l} = pu^w + q \land u \ge 2^{z_l} \land u^w > q \land y_1 \mid u \land y_1 \mid p].
\end{split}
\end{equation*}
The conditions $u \ge 2^{z_l}$, $u^w > q$, $y_1 \mid y$, and $y_1 \mid p$ introduce 4 more variables, bringing the total again, to 7 new variables. However, we can reduce the total number of new variables introduced, namely by eliminating $u$ and $y_1 \mid u$ by using the equation
\begin{equation} \label{s5}
u = ({y_1}+1)2^{z_l}.
\end{equation}
Thus, we get our result.
\end{proof}
So, the representation of \eqref{s1} has been reduced from introducing 7 new variables, to introducing only 4 new variables. Hence, already for $m \ge 7$, we see that the number of variables conserved by \eqref{s2} supercedes the number of variables conserved by the strong inequality in \eqref{save2}.

\section{Exponential prefixed polynomial representations of independent statements}
\subsection{Proof of Theorem ~\ref{result1}}
Bovykin and De Smet's \cite{bovykin} intermediate representation of Theorem ~\ref{ph2} in prefixed polynomial form (involving alternations of existential and universal quantifiers) is as follows:
\begin{equation} \label{e1}
\begin{split}
&\forall k \mbox{            } \exists M \mbox{            } \forall ab \mbox{            } \exists cdAX \mbox{            } \forall xy \mbox{            } \exists BCF \\
& [( 0 < x \land x < y \land y \le A + k - 1) \to \\
& (A = \operatorname{rem}(c, d + 1) \land B = \operatorname{rem}(c, dx + 1) \land C = \operatorname{rem}(c, dy + 1) \land B < C \\
& \land C < M + 1 \land F = \operatorname{rem}(a, b(B + C^2) + 1) \land F \equiv X\mod{r})].
\end{split}
\end{equation}
Note that eliminating the bounded quantifiers from the intermediate representation of Theorem ~\ref{ph2} rather than directly from the full prefixed polynomial in \eqref{ppr1} results in smaller exponential Diophantine representation for PH$^2$. Our notation is pretty similar to \cite{bovykin}, but to clarify, $(a, b)$ codes $f$ and $(c,d)$ codes the homogeneous set $Y$ (and $x$ and $y$ are indices of elements coded by $(c,d)$). Moreover, $A$ is the first element of $Y$, namely $\operatorname{min}(Y)$, since $Y$ is ordered.\newline
\indent Note that we will be taking $k, M, a, b$, and $r$ as parameters (thus, the $\forall k$ and $\forall ab$ in \eqref{e1} pose no difficulty). Further, note that since $x$ is a natural number $x \ge 0$, but in \eqref{e1}, $x > 0$. To fix this, we simply have to modify $B$ and $C$ to $B = \operatorname{rem}(c, d(x + 1) + 1)$ and $C = \operatorname{rem}(c, d(y + 1) + 1)$.  Second, since we have updated $x$ and $y$ to be $x+1$ and $y+1$, respectively, and $x < y$, then $\forall xy \le A + k - 1$ is incorrect. Instead, we bound $x$ and $y$ as $\forall x \le A + k - 3 \mbox{         } \forall y \le A + k - 2$. Hence, the representation we are left to deal with is:
\begin{equation} \label{e3}
\begin{split}
&\exists cdAX \mbox{            } \forall x \le A + k - 3 \mbox{              } \forall y \le A + k - 2 \mbox{            } \exists BCF \\
& [x < y \land A = \operatorname{rem}(c, d + 1) \land B = \operatorname{rem}(c, d(x+1) + 1) \land C = \operatorname{rem}(c, d(y+1) + 1) \land B < C \\
& \land C < M + 1 \land F = \operatorname{rem}(a, b(B + C^2) + 1) \land F \equiv X\mod{r}].
\end{split}
\end{equation}
Expanding \eqref{e3}, we have
\begin{equation} \label{e4}
\begin{split}
&\exists cdAX \mbox{            } \forall x \le A + k - 3 \mbox{              } \forall y \le A + k - 2 \mbox{            } \exists BCFv_1,\ldots,v_{17} \\
& [(x+v_1+1 - y)^2 + ((A+v_2 - d)^2 + ((d+1)v_3 - c + A)^2)^2 \\
&+ ((B+v_4-d(x+1))^2 + ((d(x+1)+1)v_5 - c + B)^2)^2 \\
&+ ((C+v_6-d(y+1))^2 + ((d(y+1)+1)v_7 - c + C)^2)^2 \\
&+ (B+v_8+1 - C)^2 + (C+v_9 - M)^2 \\
&+ ((F+v_{10}-b(B+C^2))^2 + ((b(B+C^2)+1)v_{11} - a + F)^2)^2 \\
&+ ((v_{12} + v_{14} + 1 - r)^2 + (rv_{15} - F + v_{12})^2)^2 \\
&+ ((v_{13} + v_{16} + 1 - r)^2 + (rv_{17} - X + v_{13})^2)^2 + (v_{12} - v_{13})^2 = 0]. 
\end{split}
\end{equation}
We can reduce the two bounded quantifiers in \eqref{e4} to just one by taking advantage of the fact that if $x \le A+ k - 3$ and $y \le A+k -2 $, then $J(x, y) \le J(A+k-3, A+k-2)$, where $J$ is Cantor's function defined for natural numbers $m$ and $n$ as $J(m,n) = \frac{1}{2}((m+n)^2 + 3m + n)$. Thus, we have
\begin{equation} \label{e6}
\begin{split}
& \exists cdAXz \mbox{            } \forall t < z + 1 \mbox{            } \exists xyBCFv_1,\ldots,v_{19} \\
& [(2z - (2A+2k-5)^2 - 4A - 4k + 11)^2 + (2t - (x+y)^2 - 3x - y)^2 \\
&+ ((A+ k - 2 + v_{18} - x) \cdot (A+k - 1 + v_{19} - y)\cdot((x+v_1+1 - y)^2 \\
&+ ((A+v_2 - d)^2 + ((d+1)v_3 - c + A)^2)^2 \\
&+ ((B+v_4-d(x+1))^2 + ((d(x+1)+1)v_5 - c + B)^2)^2 \\
&+ ((C+v_6-d(y+1))^2 + ((d(y+1)+1)v_7 - c + C)^2)^2 \\
&+ (B+v_8+1 - C)^2 + (C+v_9 - M)^2 \\
&+ ((F+v_{10}-b(B+C^2))^2 + ((b(B+C^2)+1)v_{11} - a + F)^2)^2 \\
&+ ((v_{12} + v_{14} + 1 - r)^2 + (rv_{15} - F + v_{12})^2)^2 \\
&+ ((v_{13} + v_{16} + 1 - r)^2 + (rv_{17} - X + v_{13})^2)^2 + (v_{12} - v_{13})^2))^2 = 0]. 
\end{split}
\end{equation}
The final removal of the remaining bounded universal quantifier $\forall t < z+1$ will be done explicitly by the methods presented in \S 2. Based on \eqref{e6}, we then define the following polynomial $P(k, M, a, b, r, c, d, A, X, z, z_0, z_1,\ldots,z_{24})$ as:
\begin{equation*}
\begin{split}
& (2z - (2A+2k-5)^2 - 4A - 4k + 11)^2 + (2z_0 - (z_1+z_2)^2 - 3z_1 - z_2)^2 \\
&+ ((A+ k - 2 + z_{23} - z_1) \cdot (A+k - 1 + z_{24} - z_2)\cdot((z_1+z_6+1 - z_2)^2 \\
&+ ((A+z_7 - d)^2 + ((d+1)z_8 - c + A)^2)^2 \\
&+ ((z_3+z_9-d(z_1+1))^2 + ((d(z_1+1)+1)z_{10} - c + z_3)^2)^2 \\
&+ ((z_4+z_{11}-d(z_2+1))^2 + ((d(z_2+1)+1)z_{12} - c + z_4)^2)^2 \\
&+ (z_3+z_{13}+1 - z_4)^2 + (z_4+z_{14} - M)^2 \\
&+ ((z_5+z_{15}-b(z_3+z_4^2))^2 + ((b(z_3+z_4^2)+1)z_{16} - a + z_5)^2)^2 \\
&+ ((z_{17} + z_{19} + 1 - r)^2 + (rz_{20} - z_5 + z_{17})^2)^2 \\
&+ ((z_{18} + z_{21} + 1 - r)^2 + (rz_{22} - X + z_{18})^2)^2 + (z_{17} - z_{18})^2))^2. 
\end{split}
\end{equation*}
Again based on \eqref{e6}, we define (with some simplifications) the polynomial $B(k, M, a, b, r, c, d, A, X, z, w)$ as:
\begin{equation*}
\begin{split}
&(1 + A + k + 2 w)^2 (2 + A + k + 2 w)^2 (4 w^2 + (M + 2 w)^2 + 2 (1 + 3 w)^2 \\
& + ((2 + r)^2 w^2 + (1 + r + 2 w)^2)^2 + ((A + d + w)^2 + (A + c + w + d w)^2)^2 \\
& + 2 ((d + 2 w + d w)^2 + (c + w (2 + d + d w))^2)^2 + (w^2 (2 + b + b w)^2 \\
& + (a + w (2 + b w (1 + w)))^2)^2 + ((1 + r + 2 w)^2 + (w + r w + X)^2)^2)^2 \\
& + 4 (1 + 2 w + 2 w^2 + z)^2 + (11 + 4 A + 4 k + (5 + 2 A + 2 k)^2 + 2 z)^2.
\end{split}
\end{equation*}
Thus, \eqref{e6} (and by consequence, \eqref{e3}) is equivalent to the following system of 27 Diophantine conditions solvable in the unknowns $c, d, A, X, z, q, w, z_0, \ldots, z_{24}$ (with $k, M, a, b, r$ as parameters):
\begin{equation} \label{e7}
\begin{split}
& P(k, M, a, b, r, c, d, A, X, z, z_0, z_1,\ldots,z_{24}) \equiv 0 \mod{\binom{q}{z+1}}, \\
& z_0 = q, \\
& q > z+1 + (z+2)^{z+2}((z+2)^{z+2}B(k, M, a, b, r, c, d, A, X, z, w))^{w^{24}},\\
& \binom{q}{z+1} \mid \binom{z_1}{w},\\
&\hspace{0.55 in}\vdots\\
& \binom{q}{z+1} \mid \binom{z_{24}}{w}.
\end{split}
\end{equation}
\indent We proceed as follows. Letting $l$ be a dummy variable used only for indexing, we introduce the variables $y_1, y_2, j_1,\ldots,j_{3}, f_1,\ldots,f_{24}, g_1,\ldots,g_{24}, m_1,\ldots,m_{24}, s_1,\ldots, s_{24}, h_1, \ldots, h_{5}$ to represent the following equivalences:
\begin{equation} \label{e10}
\begin{split}
& y_1 = \binom{q}{z+1} \\
& \Leftrightarrow (z+1 - h_1)^2 + \left((2^q + 2)^q - h_2(2^q + 1)^{z+2} - y_1(2^q + 1)^{h_3} - h_3\right)^2 \\
& + \left(y_1 - h_4 - 2^q\right)^2 + \left(h_3 + h_5 + 1 - (2^q + 1)^{h_1}\right)^2 = 0, \\
& P(k, M, a, b, r, c, d, A, X, z, z_0, \ldots z_{24}) \equiv 0 \mod{y_1} \\
& \Leftrightarrow (j_1 + j_2 + 1 - y_1)^2 + (y_1j_3 - P(k, M, a, b, r, c, d, A, X, z, z_0, \ldots z_{24}) + j_1)^2 + j_1^2 = 0, \\
& z_0 = q \\
& \Leftrightarrow z_0-q = 0, \\
& q > z+1 + (z+2)^{z+2}((z+2)^{z+2}B(k, M, a, b, r, c, d, A, X, z, w))^{w^{24}},\\
& \Leftrightarrow z + 2 + (z+2)^{z+2}((z+2)^{z+2}B(k, M, a, b, r, c, d, A, X, z, w))^{w^{24}} + y_2 - q = 0,\\
& \text{For $l = 1, \ldots, 24$}, \\
& y_1 \mid \binom{z_l}{w} \\
& \Leftrightarrow \left(\left(({y_1}+1)2^{z_l} + 1\right)^{z_l} - f_{l}(({y_1}+1)2^{z_l})^{w} - g_{l}\right)^2 \\
& + \left(g_{l} + m_{l} + 1 - \left(({y_1}+1)2^{z_l}\right)^w\right)^2 + \left(y_1s_{l} - f_{l}\right)^2 = 0.
\end{split}
\end{equation}
From \eqref{e10}, we derive our desired exponential Diophantine representation solvable in the 138 unknowns $c, d, A, X, z, q, w, z_0, z_1, \ldots, z_{24}, y_1, y_2, j_1,\ldots,j_{3}, f_1,\ldots,f_{24}, g_1,\ldots,g_{24}, m_1,\ldots,m_{24}, s_1,\ldots, s_{24}, h_1, \ldots, h_{5}$ (with $k, M, a, b, r$ as parameters):
\begin{equation*}
\begin{split}
& \big((z+1 - h_1)^2 + \left((2^q + 2)^q - h_2(2^q + 1)^{z+2} - y_1(2^q + 1)^{h_3} - h_3\right)^2 + \left(y_1 - h_4 - 2^q\right)^2 + \left(h_3 + h_5 + 1 - (2^q + 1)^{h_1}\right)^2\bigr)^2 \\
&+ \big((j_1 + j_2 + 1 - y_1)^2 + (y_1j_3 - ((2z - (2A+2k-5)^2 - 4A - 4k + 11)^2 + (2z_0 - (z_1+z_2)^2 - 3z_1 - z_2)^2 \\
&+ ((A+ k - 2 + z_{23} - z_1) \cdot (A+k - 1 + z_{24} - z_2)\cdot((z_1+z_6+1 - z_2)^2 + ((A+z_7 - d)^2 + ((d+1)z_8 - c + A)^2)^2 \\
& + ((z_3+z_9-d(z_1+1))^2 + ((d(z_1+1)+1)z_{10} - c + z_3)^2)^2 + ((z_4+z_{11}-d(z_2+1))^2 + ((d(z_2+1)+1)z_{12} - c + z_4)^2)^2 \\
&+ (z_3+z_{13}+1 - z_4)^2 + (z_4+z_{14} - M)^2 + ((z_5+z_{15}-b(z_3+z_4^2))^2 + ((b(z_3+z_4^2)+1)z_{16} - a + z_5)^2)^2 \\
&+ ((z_{17} + z_{19} + 1 - r)^2 + (rz_{20} - z_5 + z_{17})^2)^2 + ((z_{18} + z_{21} + 1 - r)^2 + (rz_{22} - X + z_{18})^2)^2 + (z_{17} - z_{18})^2))^2) + j_1)^2 + j_1^2\bigr)^2 \\
&+ (z_0-q)^2 + \big(z+2+(z+2)^{z+2}((z+2)^{z+2}((1 + A + k + 2 w)^2 (2 + A + k + 2 w)^2(4 w^2 + (M + 2 w)^2 + 2 (1 + 3 w)^2 \\
&+ ((2 + r)^2 w^2 + (1 + r + 2 w)^2)^2 + ((A + d + w)^2 + (A + c + w + d w)^2)^2 + 2 ((d + 2 w + d w)^2 + (c + w (2 + d + d w))^2)^2 \\
& + (w^2 (2 + b + b w)^2 + (a + w (2 + b w (1 + w)))^2)^2 + ((1 + r + 2 w)^2 + (w + r w + X)^2)^2)^2 + 4 (1 + 2 w + 2 w^2 + z)^2 \\
&+ (11 + 4 A + 4 k + (5 + 2 A + 2 k)^2 + 2 z)^2))^{w^{24}}+y_2-q\bigr)^2 + \sum_{l = 1}^{24}\big(\left(\left(({y_1}+1)2^{z_l} + 1\right)^{z_l} - f_{l}(({y_1}+1)2^{z_l})^{w} - g_{l}\right)^2  \\
&+ \left(g_{l} + m_{l} + 1 - \left(({y_1}+1)2^{z_l}\right)^w\right)^2 + \left(y_1s_{l} - f_{l}\right)^2\bigr)^2 = 0.
\end{split}
\end{equation*}
The use of summation notation in the above representation is permissible, and in fact is more informative than expanding it. For instance, Keijo Ruohonen \cite{rn}, uses summation notation in his 79 variable Diophantine representation of Fermat's Last Theorem in order to keep it under a page. In fact, Davis, Matiyasevich, and Robinson \cite[pg. 332]{dmr1976} acknowledge this as a writeable representation. But it is not a problem either if we were to expand the sum, as the entire representation would cover exactly half a page. \newline
\indent Now, in order to obtain a Diophantine representation of PH$^2$, it would not be sufficient to simply replace every exponential function with its Diophantine representation (say, for the sake of example, even a representation that is economical with respect to the number of variables, namely, five, as presented in \cite{mat-rob1975}), since this would result in a representation of 638 variables. Instead, if one uses the 10 variable Diophantine representation of the binomial coefficient obtained via Theorem ~\ref{t2} (and replace all the exponential functions of the sufficiently strong inequality for $q$ in \eqref{e10} with the five variable representation), then we will obtain a Diophantine representation of PH$^2$ in 347 variables (though if one is concerned with writeability and not just the theoretical minimization of variables, then the 18 variable version should be used instead). We should also mention that for the specific case of PH$^2$, the exponential Diophantine and Diophantine equations obtained could further be used via some combinatorial tricks in Bovykin and De Smet's original intermediate representation of PH$^2$ presented in \eqref{e1}, which Bovykin and De Smet are apt to point out. However, despite the naivete of the starting representation in \eqref{e1} prior to the elimination of the bounded universal quantifier, the fact that we have still obtained compact exponential Diophantine and Diophantine representations of PH$^2$ points to importance of the conservation of variables due to the techniques of \S 2. To compare, had we directly used the method of Davis, Putnam, and Robinson \cite{dpr} (namely, \eqref{eq2} and \eqref{eq3}) without any of the techniques in \S 2 to eliminate the bounded universal quantifier in \eqref{e6}, then a rough estimate would yield an exponential Diophantine representation of 233 variables and, under the assumption of a five variable representation of the exponential function, a 1055 variable Diophantine representation. Thus, we have managed to conserve 95 variables and 708 variables for each case, respectively.
\subsection{Proof of Theorem ~\ref{result2}}
Goodstein's Theorem \cite{goodstein} is the following:
\begin{theorem}[Goodstein's Theorem]
Given any non-decreasing function $p_r$, $p_0 \ge 2$, a number $n_0$, and the function $n_r$ defined as:
\begin{equation*}
n_{r+1} = S_{p_{r+1}}^{p_r}(n_r)-1,
\end{equation*}
then $\exists r. \mbox{         } n_r = 0$, where $S_{p_{r+1}}^{p_r}(n_r)$ is the operation of putting $n_r$ in hereditary base-$p_r$ notation and then replacing every occurence of $p_r$ in this representation with $p_{r+1}$.
\end{theorem}
\indent We will instead use the following simpler notation (due to Kirby and Paris \cite{kp}): Let $G_n(m)$ be the number produced by replacing every $n$ in the hereditary base-$n$ representation of $m$ by $n+1$ and subtracting 1. So, the Goodstein sequence for $m$ starting at 2 is:
\begin{equation*}
m_0 = m, \mbox{            } m_1 = G_2(m_0), \mbox{            } m_2 = G_3(m_1), \mbox{            } m_3 = G_4(m_2),\mbox{            }\ldots
\end{equation*}
Hence, $m_i$ can be viewed as the \emph{cipher} of the base $i+a$ positional code $\tuple{m_i,i+a,l+1}$ of the sequence $\tuple{c_0,\ldots,c_l}$. Thus, the statement of Goodstein's theorem can be rewritten as (where $m > 1$ and $a > 1$ are taken as parameters):
\begin{equation*}
\exists k \mbox{            } \forall i \le k \mbox{            } (m_{i+1} = G_{i+a}(m_i) \land m_0 = m \land m_k = 0).
\end{equation*}
In other words, the Goodstein sequence for $m$ starting at $a$ eventually terminates. Hence, the bulk of the work will be focused on representing the action of $G_{i+a}(m_i)$ as a prefixed polynomial expression. \newline
\indent For any $m_i$, the base $i+a$ representation of $m_i$ is
\begin{equation*}
m_i = c_l(i+a)^{l} + c_{l-1}(i+a)^{l-1} + \ldots + c_1(i+a) + c_0.
\end{equation*}
Then $G_{i+a}(m_i)$ can be defined in terms of a recursive function $f$ (which puts $m_i$ in its hereditary base $i+a$ representation and replaces every occurence of $i+a$ with $i+1+a$) as such:
\begin{equation*}
G_{i+a}(m_i) = f^{m_i,i+a}(i+1+a)-1,
\end{equation*}
where
\begin{equation} \label{f}
f^{m_i,i+a}(i+1+a) = \sum_{j = 0}^{l} c_j(i+1+a)^{f^{j,i+a}(i+1+a)},
\end{equation}
where for each $c_j$ ($0 \le j \le l$) has the following Diophantine representation, namely,
\begin{equation*}
c_j = \Elem(m_i,i+a,j).
\end{equation*}
Note that by definition,
\begin{equation*}
f^{0,i+a}(i+1+a) := 0.
\end{equation*}
So the next step is to construct a prefixed polynomial expression for $f$. The first step is to view $f$, the process of putting a number in hereditary base notation, in terms of levels that each terminate at the highest power of the level. Note that the highest power $l$ of the base $i+a$ representation of a natural number $n$ can be given as the following Diophantine representation:
\begin{equation*}
l = \HP(n) \Leftrightarrow \left((i+a)^{l+1} > n \land n \ge (i+a)^l \land n \ne 0\right) \lor \left(n = 0 \land l = 0\right).
\end{equation*}
I suspect though that there must be a number-theoretic function that already does this, so I do not need as many variables in the representation of $l = \HP(n)$ (however, for now, this will have to suffice). With this definition in hand, we can proceed to define $\RB(n)$, which returns the natural number which results from taking the base $i+a$ representation of $n$ and replacing every occurence of $i+a$ with $i+1+a$:
\begin{equation*}
\begin{split}
n' & = \RB(n) \Leftrightarrow \\
&\exists l \mbox{         } \mbox{         } \forall k \le l \mbox{         } \exists c \mbox{            }[l = \HP(n) \land a_0 = \Elem(n,i+a,0) \land a_{k+1} = a_{k} + c \cdot (i+a+1)^{k+1} \\
& \land c = \Elem(n,i+a,k+1) \land a_l = n'].
\end{split}
\end{equation*}
Note that the sequence $a_0,\ldots,a_l$ will be G\"{o}del coded (by the pair $b,d$) so the above definition of $\RB$ formally becomes
\begin{equation*}
\begin{split}
n' & = \RB(n) \Leftrightarrow \\
&\exists l b d \mbox{         } \forall k \le l \mbox{         } \exists c \mbox{         } [l = \HP(n) \land \rem(b,1+d) = \Elem(n,i+a,0) \\
&\land \rem(b,1+(k+2)d) = \rem(b,1+(k+1)d) + c \cdot (i+a+1)^{k+1} \\
& \land c = \Elem(n,i+a,k+1) \land \rem(b,1+(l+1)d) = n'].
\end{split}
\end{equation*}
For clarity of exposition, I will leave sequences be and not replace them with the G\"{o}del code of their elements via the remainder function. So, $f$ can be viewed as the following sequence, $s_0,s_1,\ldots,s_L$ where
\begin{equation} \label{s}
\begin{split}
s_0 &= \operatorname{RB}(m_i,\tuple{0,\ldots,\HP(m_i)})\\
s_1 &= \operatorname{RB}(m_i,\tuple{\operatorname{RB}(0,\tuple{0,\ldots,\HP(0)}),\ldots,\operatorname{RB}(\HP(m_i),\tuple{0,\ldots,\HP(\HP(m_i))})}) \\
\vdots \\
s_{L} &= m_{i+1}+1.
\end{split}
\end{equation}
Namely, $s_0$ is the $\RB$ operation applied to $m_i$ and $\tuple{0,\ldots,\HP(m_i)}$ are the exponents of $s_0$. Next, $s_1$ is the natural number that results from taking $s_0$ and applying the $\RB$ operation to its exponents. $s_2$ is the natural number that results from taking $s_0$ and applying the $\RB$ operation to its exponents. And so forth, until we terminate at $s_L$, which gives us the natural number $m_{i+1}+1$ (since when we subtract 1 we get $m_{i+1}$). First, we can determine $L$, more exactly as:
\begin{equation*}
\begin{split}
L &= \operatorname{level}(m_i) \Leftrightarrow \\
& \forall k \le L [a_0 = m_i \land a_{k+1} = \HP(a_k) \land a_L < i+a \land a_{L-1} \ge i+a].
\end{split}
\end{equation*}
It is also important to be able to access each of the exponents of any $s_n$ in the sequence defined in \eqref{s} because then we can apply the $\RB$ operation on them. Thus, the property ``$p$ is the $k$-th exponent of $s_n$'' is Diophantine:
\begin{equation*}
\begin{split}
p &= \operatorname{Exp}_k(s_n) \Leftrightarrow \\
& \exists cy \mbox{       } [c(i+1+a)^{p} + y = s_n \land c = \Elem(m_i,i+a,k)].
\end{split}
\end{equation*}
With these two representations, the function $f^{m_i,i+a}(i+1+a)$ (as defined in \eqref{f}) can be represented as the following prefixed polynomial (note that we use the following three sequences and omit the explicit use of the remainder function to represent their elements for clarity of presentation: $z_0,\ldots,z_L$, $s_0,\ldots,s_L$, $\tau_0,\ldots,\tau_l$):
\begin{equation*}
\begin{split}
& \exists L \mbox{        }\forall n \le L \mbox{         } \exists l \mbox{         } \forall k \le l \mbox{         } [z_0 = m_i \land z_{n+1} = \HP(z_n) \land z_L < i+a \land z_{L-1} \ge i+a \\
&\land s_0 = \RB(m_i) \land l = z_1 \land \tau_0 = c\cdot (i+1+a)^{d} \land c = \Elem(m_i,i+a,0) \\
&\land d = \RB(d') \land d' = \operatorname{Exp}_0(s_n) \land \tau_{k+1} = \tau_k + f \cdot (i+1+a)^{d'''} \\
& \land f = \Elem(m_i,i+a,k+1) \land d''' = \RB(d'') \land d'' = \operatorname{Exp}_{k+1}(s_n) \land \tau_l = s_{n+1} \land s_L = m_{i+1}+1].
\end{split}
\end{equation*}
\newline
\newline
Thus, the statement of Goodstein's theorem is:
\begin{equation*}
\begin{split}
& \exists r\mbox{        } \forall i \le r \mbox{        } \exists L \mbox{        }\forall n \le L \mbox{         } \exists l \mbox{         } \forall k \le l \mbox{         } [z_0 = m_i \land z_{n+1} = \HP(z_n) \land z_L < i+a \land z_{L-1} \ge i+a \\
&\land s_0 = \RB(m_i) \land \tau_0 = c\cdot (i+1+a)^{d} \land c = \Elem(m_i,i+a,0) \\
&\land d = \RB(d') \land d' = \operatorname{Exp}_0(s_n) \land \tau_{k+1} = \tau_k + f \cdot (i+1+a)^{d'''} \\
& \land f = \Elem(m_i,i+a,k+1) \land d''' = \RB(d'') \land d'' = \operatorname{Exp}_{k+1}(s_n) \land \tau_l = s_{n+1} \\
& \land s_L = m_{i+1}+1 \land m_0 = m \land m_r = 0].
\end{split}
\end{equation*}
\newline
Expanding this, we get the following 181 variable exponential prefixed polynomial representation:
\newpage
\begin{equation*}
\begin{split}
&\exists rb_1d_1\mbox{        } \forall i \le r \mbox{        } \exists v_1v_{14}v_{15}v_{16}h_1h_2h_3h_4w_3w_4p_1p_2L \mbox{        }\forall n \le L \mbox{         } \exists v_2v_3v_4v_5v_6v_8v_{12}v_{13}w_1w_2w_5w_6w_7w_8w_9w_{10}w_{11}w_{12}w_{13}w_{14}w_{15}w_{16}\\
&w_{17}j_1j_2j_3j_4j_5j_6e_1e_2u_1u_2u_3 \mbox{         }\forall u_5 \le u_1 \mbox{          } \exists u_4u_6\ldots u_{30}l \mbox{         } \forall k \le l \mbox{         } \exists v_7cdd'v_9v_{10}fd''d'''h_5h_6h_7h_8h_9h_{10}h_{11}h_{12}h_{13}h_{14}\mbox{         }\forall h_{15} \le h_{12}\mbox{      }\\
&\exists h_{16}\ldots h_{41}j_7\ldots j_{24}\mbox{        }\forall j_{25} \le j_{22} \mbox{       }\exists j_{26}\ldots j_{51}q_1 \ldots q_9 \mbox{       }\\
&[((v_1 + w_3 - d_1(i+1))^2 + ((d_1(i+1)+1)w_4 - b_1 + v_1)^2)^2 + ((v_2 + w_5 - d_3(n+2))^2 + ((d_3(n+2)+1)w_6 - b_3 + v_2)^2)^2 \\
&+ ((v_3 + w_{10} - d_3(n+1))^2 + ((d_3(n+1)+1)w_{11} - b_3 + v_3)^2)^2 + ((v_1 + w_1 - d_3)^2 + ((d_3+1)w_2 - b_3 + v_1)^2)^2 \\
&+ ((((v_3 + w_7 + 1 - (i+a)^{v_2+1})^2 + ((i+a)^{v_2})^2 +w_8 - v_3)^2  + (v_3^2-w_9-1)^2) \cdot (v_3^2 + v_2^2))^2 + ((v_4 + w_{12} - d_3(L+1))^2 \\
&+ ((d_3(L+1)+1)w_{13} - b_3 + v_4)^2)^2 + (v_4 +w_{14} + 1 - i - a)^2 + ((v_5 + w_{15} - d_3L)^2 + ((d_3L+1)w_{16} - b_3 + v_5)^2)^2 \\
&+ (i+a +w_{17} + 1 - v_5)^2 + ((v_6 + j_1 - d_2)^2 + ((d_2+1)j_2 - b_2 + v_6)^2)^2 + ((u_6 + u_{10} - u_3)^2 + ((u_3+1)u_{11} - u_2 + u_6)^2)^2 \\
&+ ((u_{14} + 1)^2 + (u_{12} + u_6(i+a)^{u_{14}} + u_{13} - v_1)^2 + (u_6 + u_{15} + 1 - (i+a))^2 + (u_{13} + u_{16} + 1 - (i+a)^{u_{14}})^2)^2 \\
&+ ((u_7 + u_{17} - (u_5+2)u_3)^2 + (((u_5+2)u_{3}+1)u_{18} - u_2 + u_7)^2)^2 + ((u_8 + u_{19} - (u_5+1)u_3)^2 + (((u_5+1)u_{3}+1)u_{20} \\
&- u_2 + u_8)^2)^2 + (u_8 + u_4(i+a+1)^{u_5+1} - u_7)^2 + ((u_5 - u_{23})^2 + (u_{21}(i+a)^{u_5+1} + u_4(i+a)^{u_{23}} + u_{22} - u_4)^2 \\
&+ (u_4 + u_{24} + 1 - (i+a))^2 + (u_{22} + u_{25} + 1 - (i+a)^{u_{23}})^2)^2 + ((u_9 + u_{26} - (u_1+1)u_3)^2 + (((u_1+1)u_{3}+1)u_{27} \\
&- u_2 + u_9)^2)^2 + (u_9 - u_6)^2 + ((((v_1 + u_{28} + 1 - (i+a)^{u_1+1})^2 + ((i+a)^{u_1})^2 +u_{29} - v_1)^2 + (v_1^2-u_{30}-1)^2) \cdot (v_1^2 + u_1^2))^2)^2 \\
&+ ((v_7 + h_5 - d_4)^2 + ((d_4+1)h_6 - u_2 + v_7)^2)^2 + (v_7 - c\cdot (i+1+a)^{d})^2 + ((h_9 + 1)^2 + (h_7 + c(i+a)^{h_9} + h_8 - v_1)^2 \\
&+ (c + h_{10} + 1 - (i+a))^2 + (h_8 + h_{11} + 1 - (i+a)^{h_9})^2)^2 + ((h_{17} + h_{21} - h_{14})^2 + ((h_{14}+1)h_{22} - h_{13} + h_{17})^2)^2 \\
&+ ((h_{25} + 1)^2 + (h_{23} + h_{17}(i+a)^{h_{25}} + h_{24} - d')^2 + (h_{17} + h_{26} + 1 - (i+a))^2 + (h_{24} + h_{27} + 1 - (i+a)^{h_{25}})^2)^2 \\
&+ ((h_{18} + h_{28} - h_{14}(h_{15}+2))^2 + ((h_{14}(h_{15}+2)+1)h_{29} - h_{13} + h_{18})^2)^2 + ((h_{19} + h_{30} - h_{14}(h_{15}+1))^2 \\
&+ ((h_{14}(h_{15}+1)+1)h_{31} - h_{13} + h_{19})^2)^2 + (h_{19} + h_{16}(i+a+1)^{h_{15}+1} - h_{18})^2 + ((h_{15} - h_{34})^2 + (h_{32}(i+a)^{h_{15}+1} \\
&+ h_{16}(i+a)^{h_{34}} + h_{33} - d')^2 + (h_{16} + h_{35} + 1 - (i+a))^2 + (h_{33} + h_{36} + 1 - (i+a)^{h_{34}})^2)^2 + ((h_{20} + h_{37} - h_{14}(h_{12}+1))^2 \\
&+ ((h_{14}(h_{12}+1)+1)h_{38} - h_{13} + h_{20})^2)^2 + (h_{20} - d)^2 + ((((d' + h_{39} + 1 - (i+a)^{h_{12}+1})^2 + ((i+a)^{h_{12}})^2 +h_{40} - d')^2  \\
&+ ((d')^2-h_{41}-1)^2) \cdot ((d')^2 + h_{12}^2))^2)^2 + ((v_8 + e_1 - (n+1)d_2)^2 + ((d_2(n+1)+1)e_2 - b_2 + v_8)^2)^2 + ((j_7(i+1+a)^{d'} \\
&+ j_8 - v_8)^2 + ((j_{11} + 1)^2 + (j_9 + j_7(i+a)^{j_{11}} + j_{10} - v_1)^2 + (j_7 + j_{12} + 1 - (i+a))^2 + (j_{10} + j_{13} + 1 - (i+a)^{j_{11}})^2)^2)^2 \\
&+ ((v_9 + j_{14} - (k+1)d_4)^2 + ((d_4(k+1)+1)j_{15} - b_4 + v_9)^2)^2 + ((v_{10} + j_{16} - (k+2)d_4)^2 + ((d_4(k+2)+1)j_{17} - b_4 \\
&+ v_{10})^2)^2 + (v_9 + f \cdot (i+1+a)^{d'''} - v_{10})^2 + ((j_{19} - k)^2 + (j_{17}(i+a)^{k+1} + f(i+a)^{j_{19}} + j_{18} - v_1)^2 + (f + j_{20} + 1 - (i+a))^2 \\
&+ (j_{18} + j_{21} + 1 - (i+a)^{j_{19}})^2)^2 + ((j_{27} + j_{31} - j_{24})^2 + ((j_{24}+1)j_{32} - j_{23} + j_{27})^2)^2 + ((j_{35} + 1)^2 + (j_{33} + j_{27}(i+a)^{j_{35}} + j_{34} \\
&- d'')^2 + (j_{27} + j_{36} + 1 - (i+a))^2 + (j_{34} + j_{37} + 1 - (i+a)^{j_{35}})^2)^2 + ((j_{28} + j_{38} - j_{24}(j_{25}+2))^2 + ((j_{24}(j_{25}+2)+1)j_{39} \\
&- j_{23} + j_{28})^2)^2 + ((j_{29} + j_{40} - j_{24}(j_{25}+1))^2 + ((j_{24}(j_{25}+1)+1)j_{41} - j_{23} + j_{29})^2)^2 + (j_{29} + j_{26}(i+a+1)^{j_{25}+1} - j_{28})^2 \\
&+ ((j_{25} - j_{44})^2 + (j_{42}(i+a)^{j_{25}+1} + j_{26}(i+a)^{j_{44}} + j_{43} - d'')^2 + (j_{26} + j_{45} + 1 - (i+a))^2 + (j_{43} + j_{46} + 1 - (i+a)^{j_{44}})^2)^2 \\
&+ ((j_{30} + j_{47} - j_{24}(j_{22}+1))^2 + ((j_{24}(j_{22}+1)+1)j_{48} - j_{23} + j_{30})^2)^2 + (j_{30} - d''')^2 + ((((d'' + j_{49} + 1 - (i+a)^{j_{22}+1})^2 \\
&+ ((i+a)^{j_{22}})^2 + j_{50} - d'')^2  + ((d'')^2-j_{51}-1)^2) \cdot ((d'')^2 + j_{22}^2))^2)^2 + ((v_{12} + q_8 - (l+1)d_4)^2 + ((d_4(l+1)+1)q_9 \\
&- b_4 + v_{12})^2)^2 + ((q_1(i+1+a)^{d''} + q_2 - v_8)^2 + ((q_5 - k)^2 + (q_3(i+a)^{k+1} + q_1(i+a)^{q_5} + q_4 - v_1)^2 + (q_1 + q_6 + 1 \\
&- (i+a))^2 + (q_4 + q_7 + 1 - (i+a)^{q_5})^2)^2)^2 + ((v_{12} + j_3 - d_2(n+2))^2 + ((d_2(n+2)+1)j_4 - b_2 + v_{12})^2)^2 + ((v_{13} + j_5 \\
&- d_2(L+1))^2 + ((d_2(L+1)+1)j_6 - b_2 + v_{13})^2)^2 + (v_{13} - v_{14} - 1)^2 + ((v_{14} + p_1 - d_1(i+2))^2 + ((d_1(i+2)+1)p_2 \\
&- b_1 + v_{14})^2)^2 + (v_{14} +1 - v_{13})^2 + ((v_{15} + h_1 - d_1)^2 + ((d_1+1)h_2 - b_1 + v_{15})^2)^2 + (v_{15} - m)^2 + ((v_{16} + h_3 - d_1(r+1))^2 \\
&+ ((d_1(r+1)+1)h_4 - b_1 + v_{16})^2 + v_{16}^2)^2 = 0].
\end{split}
\end{equation*}

\section*{Acknowledgements}
The author expresses his gratitude to Professor Grigori Mints for suggesting the topic to him as well as for his guidance during the process of this research.

\newpage
\bibliographystyle{amsplain}

\end{document}